\documentclass[11pt,a4paper]{article}

\usepackage[utf8]{inputenc}
\usepackage[T1]{fontenc}
\usepackage{lmodern}
\usepackage[english]{babel}
\usepackage{amsmath,amssymb}
\usepackage{geometry}
\usepackage{setspace}
\usepackage{indentfirst}
\usepackage{titlesec}
\usepackage[english]{babel}
\usepackage{url}
\usepackage[square,numbers]{natbib}
\usepackage{enumerate}
\usepackage{graphicx}
\usepackage{xcolor}
\usepackage{float}
\usepackage{amsthm}
\usepackage{url}

\titleformat{\section}
  {\normalfont\Large\bfseries}
  {\thesection}
  {1em}
  {}

\newtheorem{theorem}{Theorem}
\newtheorem{lemma}{Lemma}

\newtheorem{corollary}{Corollary}

\begin{document}

\begin{center}
{\LARGE
Dominance and symmetry-breaking rules\\ for the Graph Burning Problem\footnote{Preprint submitted to SBPO 2026}
\par}

\vspace{0.8cm}

{\large Nice Prado, Rafael Colares\par}

\vspace{0.5cm}

{\itshape Université Clermont Auvergne, CNRS, Clermont Auvergne INP, Mines Saint-Etienne, LIMOS, 63000 Clermont-Ferrand, France\par}
{\itshape Email: niceprado2002@gmail.com, rafael.colares\_borges@uca.fr}
\end{center} 

\vspace{0.2cm}

\begin{center}
    \begin{minipage}{0.72\textwidth}
        \begin{center}
            \textbf{Abstract}
        \end{center}
        The Graph Burning Problem (GBP) is a NP-Hard combinatorial optimization problem that models the propagation of influence or contagion in a network. The propagation is represented through the metaphor of a fire spreading through the vertices of a graph. A burning process takes place in a series of discrete time-steps. At each time step, the burning process is characterized by a \textit{propagation} (where burned nodes spread the fire to their neighbors), and an \textit{ignition} (where one additional unburned node is chosen to become burned). The minimum number of steps required to burn all vertices of a graph defines its \textit{burning number}. Literature provides integer linear programming formulations to solve the problem, but with no surprise, such approaches struggle to converge as the graph size increases. Therefore, reducing the search space explored by these formulations becomes a key point to improve performances. In this work, we study the similarities of the Graph Burning Problem with the well-known Dominating Set Problem. We propose a new formulation based on this study and apply dominance rules and symmetry-breaking techniques to reduce the search space and consequently speed up resolution time. We also introduce a perturbation of the proposed objective function, as well as a pruning rule for the perturbed model in order to further accelerate its resolution.
    \end{minipage}
\end{center}

\vspace{0.1cm}

\section{Introduction} 

The study of contagion in social networks is motivated by the need to understand how influence, information, or behaviors spread through networked interactions. To capture this phenomenon, \textit{graph burning} was introduced as a simplified model of influence propagation in networks, enabling the analysis of how quickly such contagion spread across a network \citep{bonato2014burning}. Such process takes place over a sequence of time steps. In this context, the vertices of a graph are labeled as either \textit{unburned} or \textit{burned}. At the start of the process, all vertices are said to be \textit{unburned}. At each time step, the following two procedures are sequentially applied.
\begin{enumerate}[i.]
    \item \textit{Propagation} : all \textit{burned} vertices propagate the fire to their neighbors, who become \textit{burned};
    \item \textit{Ignition} : one may choose a vertex to be ignited. The chosen vertex becomes \textit{burned}.
\end{enumerate}
A vertex that is \textit{burned} cannot become \textit{unburned}. The process continues as long as there are still \textit{unburned} vertices in the graph. The Graph Burning Problem (GBP) then consists of choosing the ignition sequence of vertices that minimizes the number of time steps required to burn the whole graph. The burning number of a graph $G$, denoted $b(G)$, is said to be $k$ if it can be completely burned in $k$ time steps, but not in $k-1$.  The lower the value of $b(G)$, the easier it is to spread such contagion in $G$. 

\subsection{Literature review}

In \cite{bessy2017burning} the decision GBP is shown to be NP-complete even when the underlying graph is restricted to the class of spider graphs, path-forests and acyclic graphs with maximum degree three. However, if the number of arms and components, respectively, are fixed, they provide polynomial time algorithms for finding the $b(G)$ of spider graphs and path-forests. For general graphs, they propose a 3-approximation polynomial-time algorithm. Heuristic approaches play an important role in addressing the GBP due to its computational hardness. For instance, \cite{garcia2025greedy} propose a deterministic greedy heuristic based on a reduction to the clustered maximum coverage problem. 
The focus of this paper is nevertheless to solve GBP exactly for any given graph and hence our literature review focuses in this part of the state-of-the-art.

In this direction, \cite{bonato2021survey} surveys the main theoretical results on graph burning, focusing on bounds and algorithms related to the burning number. In particular, they give special attention to the \emph{burning number conjecture}, which claims that any connected graph $G$ with $n$ vertices can be burned in at most $\lceil \sqrt{n} \rceil$ steps. This conjecture remains open in general, although it is known to hold for some families of graph classes, such as paths, \textit{spiders}, \textit{caterpillars}, \textit{2-caterpillars}, and any graph with minimum degree $\delta \ge 23$ \citep{bonato2021survey}.

The first integer linear programming (ILP) formulations for the GBP are introduced in \citep{garciadiaz2022graph}. In \citep{pereira2024graphburning}, the authors propose an exact algorithm for the GBP that improves the results from \citep{garciadiaz2022graph}. Their work introduces a decision covering IP model, namely \texttt{GBP-IP}, that is capable of verifying whether a given graph $G$ can be burned in $k$ steps. The resolution of such a model is then embedded within a dichotomic search for the best value of $k$. Their approach makes use of the \texttt{BFF} heuristic, introduced in \citep{garcia2022burning} to provide efficient upper bounds for the $b(G)$. However, exact approaches still struggle to produce good results when the size of the graph increases. Therefore, finding ways to reduce the search space represents a central challenge.

\subsection{Our contribution}

In this paper we introduce a new ILP formulation for the GBP. In order to boost its performance, we propose new dominance rules and symmetry-breaking techniques that allow reduce the search space. In order to further reduce symmetry, we introduce a perturbation of the objective function that ensures preservation of optimality. To accelerate the solution of the perturbed model, we also propose a pruning rule that eliminates nodes that are not relevant for exploration.
Our approach allows us to identify and eliminate redundant variables, as well as to guide the search of the Branch-and-Bound enumeration tree towards promising nodes. This significantly improves the computational performances as demonstrated by our numerical experiments.

The paper is organized as follows. Section 2 presents a formal definition of GBP. Section 3 introduces the proposed mathematical formulation. In Section 4, we discuss symmetry and dominance rules in combinatorial optimization, as well as their role in the GBP. Section 5 describes two dominance rules specifically proposed for the GBP. Section 6 presents a perturbation of the objective function to further reduce the formulation's symmetry, along with a pruning rule to speed up optimality convergence. Experimental results are reported and discussed in Section 7. Finally, Section 8 concludes the paper.

\section{The Graph Burning Problem}

Let $G=(V,E)$ be a connected graph, where $V$ denotes its set of vertices and $E$ the set of edges. In this paper, we will only consider the case where $G$ is undirected but all our results can be easily extended to the directed case.

Suppose that in the burning process of $G$, vertex $v_i$ is chosen to be ignited at step $i \in \mathbb{N}$. Such vertex is called the \emph{i-th fire source}. At time step $j \geq i$, the set of vertices that become burned due to the propagation of fire source $v_i$ are the ones at distance at most $j-i$ from $v_i$, that is $\{ u \in V : d(v_i,u) \leq j-i\}$, where $d(v_i,u)$ denotes the distance between vertices $v_i$ and $u$. Notice that this corresponds to the ball of radius $j-i$ centered at $v_i$.

A valid solution to GBP is called a \textit{burning sequence}. A burning sequence is an ordered subset of vertices $\{v_1,v_2,\dots,v_k\} \subseteq V$ that can burn the whole graph by choosing $v_i$ to be ignited at step $i$, for $i=1, \dots, k$. It follows that $\{v_1,v_2,\dots,v_k\}$ is a burning sequence if and only if 
\begin{equation}\label{eq:cover_condition}
\bigcup_{i=1}^{k} B_{k-i}(v_i) = V,
\end{equation}
where $B_r(v)= \{ u \in V : d(v,u) \leq r\}$ denotes the ball of radius $r$ centered at $v$. The optimal solution to GBP is characterized by the burning sequence of shortest length. In this sense, GBP can be viewed as the \textit{covering problem} of finding the smallest integer $k$ and an ordered subset of $k$ vertices $\{v_1, \dots, v_k\}$ such that the set of balls of radius $k-i$ centered in $v_i$ is able to \textit{cover} the whole set of vertices $V$. 
\begin{figure}[H]
    \centering
    \includegraphics[width=0.50\textwidth]{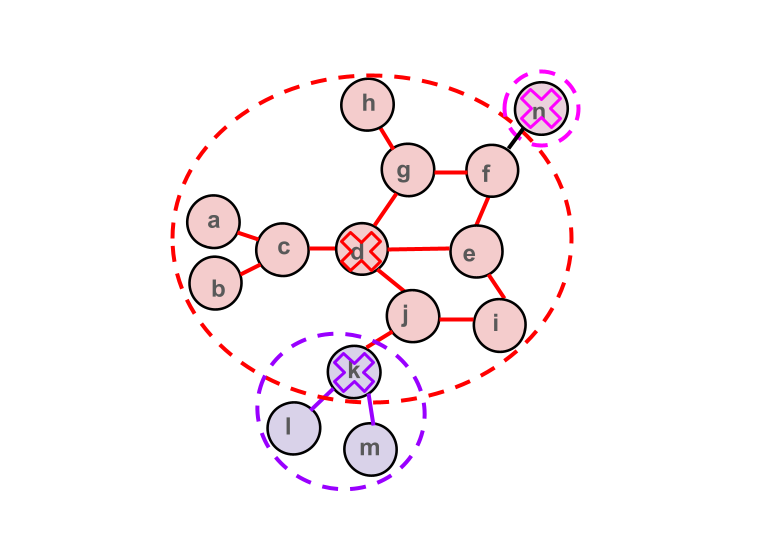}
    \caption{Illustration of burning sequence $\{d,k,n\}$. Crossed vertices correspond to ignitions and dashed zones to their associated balls.} 
    \label{fig:ex_burning_sequence}
\end{figure}

An example of burning sequence is presented in Figure \ref{fig:ex_burning_sequence}, where the entire graph is burned in three steps through the burning sequence $\{d, l, n\}$. In other words, the entire graph is covered using the associated balls $B_2(d)$, $B_1(k)$ and $B_0(n)$.

\section{Mathematical formulation}

We take the \texttt{GBP-IP} formulation proposed by \cite{pereira2024graphburning} as a starting point for our work. For completeness, we re-state it here below. 
\begin{align}
    \textnormal{Find} \quad & x_{v,i} \in \{0, 1\} & \quad & \forall v \in V, \forall i \in \{1, \dots, k\}\label{eq:per_1}\\
    \text{s.t.}\qquad   \sum_{i = 1}^k x_{v,i} &\leq 1 & \quad & \forall v \in V \label{eq:per_2} \\
    \sum_{v \in V} x_{v,i} & = 1 & \quad & \forall i \in \{1, \dots, k\} \label{eq:per_3}\\
    \sum_{i = 1}^k \sum_{u \in V : v \in B_{k-i}(u)} x_{u,i} &\geq 1 & \quad & \forall v \in V \label{eq:per_4}
\end{align}

This formulation is capable of checking whether or not there exists a burning sequence of fixed length $k$ for a given graph $G$. The binary variables $x_{v,i}$ indicate whether or not vertex $v \in V$ is the $i$-th source fire of the burning sequence for $i=1, \dots, k$. Constraints \eqref{eq:per_2} ensure that each vertex can only be ignite at most once. Constraints \eqref{eq:per_3} state that in each time step exactly one vertex is ignited. Finally, constraints \eqref{eq:per_4} guarantee that every vertex is burned within $k$ steps.

In their work, this formulation is repeatedly solved within a dichotomic search for the best value of $k$. In each iteration, \texttt{GBP-IP} is solved via a row generation algorithm that considers the covering constraints \eqref{eq:per_4} on demand. 
Next, we introduce an extended formulation based on \texttt{GBP-IP} that is capable of solving the GBP in a single run, thereby eliminating the need to embed it within a dichotomic search. It relies on selecting balls of decreasing radii that cover the entire graph $G$, while minimizing the largest radius used.
\begin{align}
    \min \quad & \tau \label{fo_1} \\ 
    \text{s.t.}\qquad \qquad \tau &\geq r \sum_{v \in V} x_{v,r}  & \quad & \forall r \in R \label{c_og_1} \\
    \sum_{v \in V} x_{v,r} &\leq 1 & \quad & \forall r \in R \label{c_og_3}\\
    \sum_{r \in R} \sum_{u \in V : v \in B_r(u)} x_{u,r} &\geq 1 & \quad & \forall v \in V \label{c_og_4}\\
    x_{v,r} &\in \{0, 1\} & \quad & \forall v \in V, \forall r \in R \label{c_og_5}\\
    \tau &\in \mathbb{Z} \label{c_og_6}
\end{align}

The binary decision variables $x_{v,r}$ are defined for every vertex $v \in V$ and available radius $r \in R$ and can be defined as follows.
\[
x_{v,r} =
\begin{cases}
1, & \text{if vertex } v \text{ is the center of a ball of radius } r, \\
0, & \text{otherwise}.
\end{cases}
\]
The set of available radii is defined as $R= \{0, \dots, UB\}$, where $UB$ denotes an upper bound on the largest radius that is required to cover the entire set of vertices. A natural choice for $UB$ is to use the upper bound derived from the \textit{Burning Number Conjecture}, that is $UB := \lceil \sqrt{|V|} \rceil$. However, this upper bound tends to be quite large when compared to $b(G)$ in general. Indeed, any good heuristic solution can be used to provide a good value for $UB$. For instance, a greedy heuristic that runs in $O(n^3)$ and yields a $(3 - 2/{b(G)})$-approximation ratio is proposed by \cite{garcia2022burning}.

The modeling variable $\tau$ corresponds to the largest radius of a ball used in the covering. Thus, the objective function aims to minimize the largest radius used and, consequently, the number of steps required to completely burn $G$. In an optimal solution, $\tau = b(G) - 1$. 

Constraints \eqref{c_og_1} indicate that $\tau$ is at least as large as the largest radius used. Constraints \eqref{c_og_3} ensure that for each value of radius, at most one ball can be selected. This implies that on each time step, at most one vertex can be ignited. Notice that, as opposed to the formulation of \citep{pereira2024graphburning}, we do not impose the ignition of a vertex at every time step. Indeed, such property can be relaxed without loss of optimality (see Lemma \ref{lemma:1}).
Constraints \eqref{c_og_4} guarantee that each vertex $v \in V$ must belong to at least one selected ball -- as otherwise, condition \eqref{eq:cover_condition} would not be satisfied.

\begin{lemma}\label{lemma:1}
    There exists an optimal solution to \eqref{fo_1}--\eqref{c_og_6} in which exactly one vertex is selected to be ignited at each time step $i \leq b(G) $.
\end{lemma}
\begin{proof}
    Suppose that in an optimal solution $(x^*, \tau^*)$ to \eqref{fo_1}-\eqref{c_og_6}, no vertex is chosen to be ignited at a given time step $1 \leq i \leq b(G)$. Let $r=b(G) - i$. This means $x_{v,r} = 0$ for all $v\in V$. If $r < b(G) - 1$, a solution of same cost can be trivially constructed by selecting any non-ignited vertex to be ignited at that time step $i$. If $r = b(G) - 1$, then the proposed solution $(x^*, \tau^*)$ is not valid from the definition of $b(G)$.
\end{proof}

Notice that our formulation also differs from the one of \citep{pereira2024graphburning} as it does not impose that a vertex can only be ignited once. Next lemma shows that if this is the case, we can easily construct a solution of same cost where each vertex is ignited at most once.

\begin{lemma}\label{lemma:2}
    There exists an optimal solution to \eqref{fo_1}--\eqref{c_og_6} in which each vertex is ignited at most once.
\end{lemma}
\begin{proof}
    Suppose that in an optimal solution to \eqref{fo_1}-\eqref{c_og_6}, a given vertex $v' \in V$ is chosen to be ignited twice. It follows that $v'$ is associated to two selected balls of different radii, say $r_1$ and $r_2$. W.l.o.g., let $r_1 < r_2 \leq \tau$. By definition, $B_{r_1}(v') \subseteq B_{r_2}(v')$, and hence a feasible solution of same cost (and that ignites $v'$ only once) can be obtained by simply setting $x_{v', r_1}$ to zero. Since only one ball can be selected at most for each radii, the new solution does not use any ball of radius $r_1$. Lemma \ref{lemma:1} we then be applied to produce a burning sequence of same length.
\end{proof}

Although the GBP can be solved in a single run with this formulation, it struggles to handle large instances efficiently. Therefore, it is necessary to investigate ways to reduce the problem's search space. To address this limitation, we explore two different techniques: dominance rules and symmetry-breaking techniques.

\section{Symmetry and dominance rules in combinatorial optimization} 

An ILP is said to be symmetric if its variables can be permuted without altering the structure of the problem \citep{margot2010symmetry}. When ILPs exhibit large symmetry groups, they become difficult to solve using traditional enumerative algorithms such as Branch-and-Bound (B\&B), even for relatively small instances. This is due to the fact that many subproblems in the enumeration tree become isomorphic, leading to unnecessary duplication of effort and slowing down convergence. 

In the case of GBP, symmetry arises from the graph structure and the dynamics associated with the fire propagation process.
Indeed, if the graph contains structurally equivalent vertices, they can be selected interchangeably as fire sources without affecting the outcome of the process. The uniform propagation of the fire implies that such choices lead to symmetric sets of burned vertices at each step, and therefore to different but equivalent burning sequences from the perspective of the objective function. 
Figure~\ref{fig:ex_symmetry} illustrates a simple example in which symmetry can be easily spotted. Each of the two solutions can be obtained from the other by permuting the variables, yielding the same objective function value.

\begin{figure}[H]
    \centering
    \includegraphics[width=0.50\textwidth]{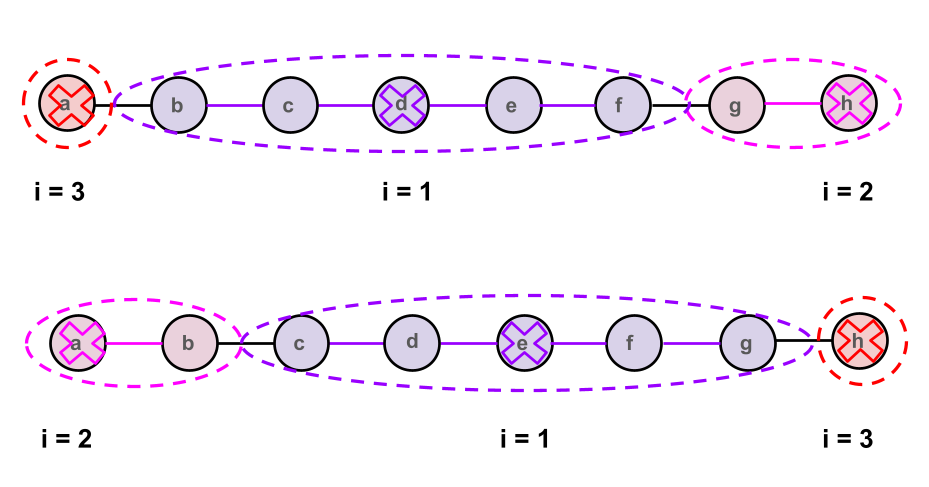}
    \caption{Example of symmetric solutions for GBP. Top solution represents the burning sequence $\{d, h, a\}$ while the bottom one is associated with $\{e, a, h\}$. Dashed zones show the vertices burned due to the ignition of the crossed vertex.} 
    \label{fig:ex_symmetry}
\end{figure}

Dominance rules play an important role in identifying ways to reduce the search space of combinatorial optimization problems. Their purpose is to identify relevant subsets of solutions within which it is sufficient to search for an optimal solution, thereby reducing the search space of combinatorial problems \citep{jouglet2011dominance}. 

While symmetry-breaking techniques are employed to ignore equivalent sets of solutions, dominance rules can be used to discard exploration of suboptimal solutions. Informally, a dominance rule consists of identifying, based on a certain property, a subset of solutions that contains at least one optimal solution. One can then solve the original problem by exploring only the reduced set of solutions. The application of a dominance rule can be integrated at different stages of the ILP solving process. Examples are variable fixing and/or elimination, introduction of (non-valid) constraints and custom management of branching rules.

An example of how efficient dominance rules are in combinatorial optimization is provided by \citep{falq2022dominance}, who derive neighborhood-based dominance inequalities that strengthen a compact MIP formulation for the Unrestrictive Common Due Date Problem, a simplified scheduling problem. Computational results show that these inequalities can drastically reduce the size of the branch-and-bound tree, with the number of explored nodes dropping from tens of thousands to only a few dozen in some instances. 

Another example appears in \citep{alber2004polynomial} where dominance rules are applied to the classic Dominating Set (DS) problem by means of a kernelization preprocessing. Notice that DS problem has its similarities with GBP. Indeed, DS can be viewed as a GBP where only balls of radius 1 are allowed, and the goal is to select the least amount of balls to cover the underlying graph. The dominance rule exploited is based on the neighborhood of vertices. Next, we use this same idea as a starting point and extend the dominance properties to the GBP.

\section{Dominance rules for the GBP}

In the case of the GBP, dominance can be exploited through the inclusion relation between balls. More specifically, let $v$ and $u$ be two vertices of $G$ and $r$ a given radius. If
\begin{equation}
    \label{eq:domination}
B_r(u) \subseteq B_r(v),
\end{equation}
then all vertices covered by $B_r(u)$ are also covered by $B_r(v)$. This means that the fire started at $u$ does not reach any additional vertex compared to the one started at $v$ within $r$ time steps. Consequently, it is reasonable to forbid the use of $B_r(u)$ if $B_r(v)$ is available.

\begin{theorem}[1DR]\label{th:1dr}
    Let $u,v \in V$ and $r \in R$ such that $B_r(u) \subset B_r(v)$. Then $x_{u,r}$ can be fixed to 0 without loss of optimality.
\end{theorem}
\begin{proof}[Proof sketch]
Let $(x^*, R^*)$ be an optimal solution to \eqref{fo_1}-\eqref{c_og_6} such that $x^*_{u,r} = 1$. Another solution can be easily constructed from $(x^*, R^*)$ by setting ${x}_{u,r} = 0$ and ${x}_{v,r} = 1$, while leaving all other variables unchanged. The new solution is feasible and has the same cost as the original one.
\end{proof}

Theorem \ref{th:1dr} can be extended to the case where the two balls cover exactly the same vertices, \textit{i.e.}, $B_r(u) = B_r(v)$. In this case, however, we need to ensure that at least one of the balls remains available. For this, we artificially impose a total order on the vertex set $V$ and allow the fixing of $x_{u,r}$ only if $u \prec v$. In this case, the rule does not eliminate a strictly dominated solution, but instead acts as a symmetry-breaking technique.

We propose exploiting \texttt{1DR} as a preprocessing step aimed at reducing the the number of decision variables before even invoking the solver. For this, we exhaustively look for every pair of vertices $u$ and $v$ in $V$, and available radius $r \in R$ that satisfy condition \eqref{eq:domination}. If $B_r(u) = B_r(v)$, we additionally verify if $u \prec v$. For all such cases, we eliminate variable $x_{u,r}$.

A natural direction is to extend this idea for balls of different radii. An interesting starting point in this direction is the observation of \citep{bonato2016burn} that a \textit{burning sequence} $(v_1, v_2, \dots, v_k)$, should satisfy
\begin{equation}
    \label{eq:domination2}
d(v_i, v_j) \geq j - i, \quad \forall\, 1 \leq i < j \leq k.
\end{equation}

This is based on the fact that if condition \eqref{eq:domination2} does not apply, it means that $v_j$ is already burning when ignited, and therefore the ignition of $v_j$ is useless. Interpreted from a dominance perspective, this suggests that balls of different radii may dominate each other depending on the distance between their centers. 

\begin{theorem}[2DR]\label{th:2dr}
    The following family of inequalities preserves optimality. 
    \begin{equation}\label{eq:2dr}
        x_{v,r'} + x_{u,r} \leq 1 \qquad \forall v \in V, u \in V, r \in R, r' \in R \textnormal{ such that } d(v, u) < r' - r 
    \end{equation}
\end{theorem}
\begin{proof}
    If $d(v, u) < r' - r$, then $B_{r}(u) \subseteq B_{r'}(v)$. It follows that if $B_{r'}(v)$ is chosen (\textit{i.e.}, $x_{v,r'} = 1$), then $B_{r}(u)$ is useless and can be safely ignored.
\end{proof}

We implement rule \texttt{2DR} by means of reinforcement cuts (\textit{i.e.}, user cuts) that are added to the formulation on demand and only when violated.

\section{Symmetry-breaking for the GBP}

The ILP defined by \eqref{fo_1}-\eqref{c_og_6} not only exhibits a high degree of symmetry, but also leads to many solutions that, while not symmetric, share the same objective function value.
It is not uncommon in linear programming for different solutions to yield the same objective value. However, in our case, this phenomenon is particularly pronounced. Indeed, the objective function depends only on the maximum radius required to cover the graph. This means that for a given value of $\tau$, a huge number of feasible solutions can usually be found. This slows down the progression of bounds during the B\&B resolution, which in its turn leads to a massive and redundant exploration of the search space.

Figure \ref{fig:motivation_perturbation} illustrates how different solutions may have same objective value using a small instance as example. Although simplified, it provides an initial intuition of the model's behavior, where all three solutions considered yield the same objective function value, namely \( \tau = 2 \). 
\begin{figure}[H]
    \centering
    \includegraphics[width=0.70\textwidth]{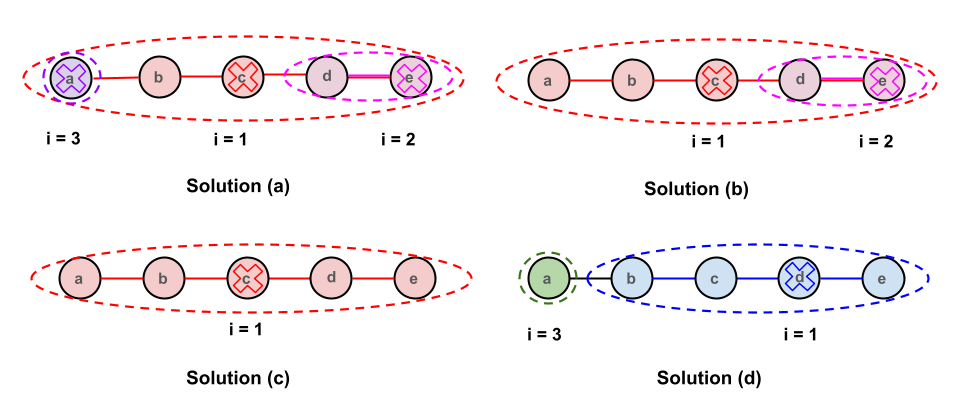}
    \caption{Different solutions yielding same value of objective function.}
    \label{fig:motivation_perturbation}
\end{figure}

To mitigate this effect and reduce the formulation's symmetry, we propose to perturb the model by replacing the objective function \eqref{fo_1} for the following one.
\begin{equation}\label{eq:pert_obj}
\min ~\ z= \sum_{r=0}^{UB}  2^r\left(\ \sum_{v \in V} x_{v,r}\ \right).
\end{equation}

With this new objective function, \textit{Solution (c)} would be the single optimal one -- with a value of $4$ -- for the instance depicted in Figure \ref{fig:motivation_perturbation}. Solutions (a), (b) and (d) would have objective values of $7$, $6$ and $5$, respectively.
The perturbation proposed relies on the following property:
\begin{equation}
\sum_{r=0}^{\bar{R}-1} 2^r = 2^{\bar{R}} - 1 < 2^{\bar{R}} \quad \textnormal{for any } \bar{R} \in \mathbb{Z_+}, \label{eq:property}
\end{equation}
which ensures that the cost associated with a given radius $\bar{R}$ exceeds the combined cost of all strictly smaller radii. Consequently, the objective value of a solution is determined primarily by its largest selected radius: any feasible solution with a strictly smaller maximum radius necessarily has a strictly smaller total cost. 

The replacement of the original objective function \eqref{fo_1} with \eqref{eq:pert_obj} also allows us to get rid of the modeling variable $\tau$ and its associated constraints \eqref{c_og_1}. The formulation defined by \eqref{eq:pert_obj}, \eqref{c_og_3}-\eqref{c_og_5} is hereafter denoted as the \textit{perturbed model}.

\begin{theorem}\label{prop:optimalite_modele_pertube}
    The perturbed model preserves the optimality of the formulation \eqref{fo_1}-\eqref{c_og_6}, in the sense that any optimal solution of the perturbed model uses a maximum radius equal to that of an optimal solution of formulation \eqref{fo_1}-\eqref{c_og_6}.
\end{theorem}
\begin{proof}
Let $x$ be a feasible solution of the perturbed model, and let $\tau(x)$ denote the maximum radius used by this solution. The objective value associated with solution $x$ can therefore be written as
\[
 z(x) = 2^{\tau(x)} + \sum_{r=0}^{\tau(x) - 1}  2^r\left(\ \sum_{v \in V} x_{v,r}\ \right).
\]

Notice that constraints \eqref{c_og_3} impose that, for any feasible solution $x$, the term inside the parentheses in the expression above is at most 1 for any $r \in R$.
Using property \eqref{eq:property}, it follows that
\[
0 \leq \sum_{r=0}^{\tau(x) - 1}  2^r\left(\ \sum_{v \in V} x_{v,r}\ \right) \leq \sum_{r=0}^{\tau(x) - 1}  2^r < 2^{\tau(x)}.
\]
In particular, this implies 
\begin{equation}
    2^{\tau(x)} \le z(x) < 2^{\tau(x)+1} \label{eq:proof}
\end{equation}
for any solution $x$ satisfying \eqref{c_og_3}-\eqref{c_og_5}.
Now, let $x^*$ be an optimal solution of the perturbed model, and let $(x', \tau^*)$ be an optimal value of formulation \eqref{fo_1}-\eqref{c_og_6}. Let us show that $\tau(x^*) = \tau^*$.

Assume by contradiction that $\tau(x^*) \ge \tau^* + 1$. Then, from \eqref{eq:proof} we have that $z(x^*) \geq 2^{\tau(x^*)} \geq 2^{\tau^*+1}$. Notice however that $x'$ is also a feasible solution of the perturbed model, and therefore $z(x') < 2^{\tau(x') + 1} = 2^{\tau^*+1}$. Finally, we have $z(x') < z(x^*)$ which contradicts the optimality of $x^*$.

On the other hand, assume now that $\tau(x^*) \le \tau^* - 1$. Notice that $(x^*, \tau(x^*))$ is also a feasible solution of formulation \eqref{fo_1}-\eqref{c_og_6}, which contradicts the optimality of $(x', \tau^*)$.
\end{proof}

\begin{corollary}
    An optimal solution of the perturbed model implicitly verifies \texttt{2DR}.
\end{corollary}

\begin{proof}
    Let $x^*$ be an optimal solution to the perturbed model such that $x^*_{v,r'} = 1$ and $x^*_{u,r} = 1$ for some $v \in V, u \in V, r \in R, r' \in R$ where $d(v,u) < r' - r$. From Theorem \ref{th:2dr}, a feasible solution $x'$ can obtained from $x^*$ by setting $x'=x^*$ and $x'_{u,r}:=0$. The obtained feasible solution has $z(x') = z(x^*) - 2^r$, which contradicts the optimality of $x^*$.
\end{proof}

Hence, no optimal solution of the perturbed model contains a redundant ball. In particular, whenever a smaller-radius ball is dominated by a larger selected ball, the objective will favor the solution that omits the smaller-radius ball, since it has strictly lower cost. In contrast, in the original formulation \eqref{fo_1}-\eqref{c_og_6}, such solutions would have the same objective value, since the objective depends only on the largest radius selected, which remains unchanged.

\subsection{Pruning rule for the perturbed model}

To accelerate the convergence of the perturbed model, we introduce a pruning rule within the B\&B algorithm that exploits the structure of the objective function.

Within the exploration of a given node $n$ in the B\&B enumeration tree, let $z^{LB}(n)$ denote the value of its linear relaxation, let $\bar{x}$ denote the current incumbent solution (\textit{i.e.}, the best integer solution known so far), and let $\tau({x})$ denote the largest radius used in any feasible solution $x$. 

\begin{theorem}\label{th:prune}
    Let $n$ be a node in the B\&B enumeration tree associated with the perturbed model. If 
    \begin{equation}\label{eq:pruning_cond}
        z^{LB}(n) \geq 2^{\tau(\bar{x})},
    \end{equation}
then node $n$ can be safely pruned.
\end{theorem}
\begin{proof}
We know that $z^{LB}(n)$ provides us a lower bound on the cost of any feasible solution $x'$ that may be found in a descendant node from $n$, that is, $z(x') \geq z^{LB}(n)$. Next we show that if $z^{LB}(n) \geq 2^{\tau(\bar{x})}$, then no feasible integer solution derived from this node can have a maximum radius strictly smaller than $\tau(\bar{x})$ (and therefore induce a shorter burning sequence than $\bar{x}$). In other words, we will show that $\tau(x') \geq \tau(\bar{x})$.

For this, let us recall that for any feasible solution $x$, $2^{\tau(x)} \leq z(x) \leq 2^{\tau(x)+1} - 1$ (see proof of Theorem \ref{prop:optimalite_modele_pertube}).
Now suppose $\tau(x') \leq \tau(\bar{x}) - 1$. Then,
\[
z^{LB}(n) \leq z(x') \leq 2^{\tau(x')+1} - 1 \leq 2^{\tau(\bar{x})} - 1 < 2^{\tau(\bar{x})},
\]
which contradicts the initial condition \eqref{eq:pruning_cond}.

Therefore, we have shown that $\tau(x') \geq \tau(\bar{x})$. This implies that node $n$ cannot lead to any improvement over the incumbent solution $\bar{x}$ and can therefore be safely pruned.
\end{proof}

We implement this rule by pruning each node of the B\&B tree whose relaxation bound $z^{LB}$ is greater than or equal to the threshold of $2^{\tau(\bar{x})}$ defined by the current incumbent solution $\bar{x}$. Notice that this pruning rule can indeed cutoff optimal solutions of the perturbed model. However, none of these cutoff solutions would lead to a shorter burning sequence. To illustrate it, consider again the simple case depicted in Figure \ref{fig:ex_symmetry}. If the incumbent solution $\bar{x}$ refers to \textit{Solution (a)}, then all other depicted solutions could be cutoff. Nevertheless, optimality is not lost with respect to the length of the associated burning sequence, and hence we maintain an exact resolution of GBP.

\section{Experimental Results}

In this section, we present the computational results obtained on the set of instances introduced in \citep{pereira2024graphburning}, restricted to graphs with at most 10,000 nodes. This yielded a test-set of 48 instances.

All implementations were carried out in \texttt{C++}, while the modeling and solution of the ILP were performed using the {Concert Technology} API from \texttt{CPLEX 22.1.2}. 
Reinforcement cuts from \texttt{2DR} are implemented using a {Generic Callback} within the {Relaxation} Context. The associated constraints are dynamically added as {user cuts} whenever a linear relaxation solution violates them. The pruning rule associated with Theorem \ref{th:prune} is implemented via a {Generic Callback} in the {relaxation} context.
Experiments were conducted on a \texttt{Dell Inspiron 13 5330} laptop equipped with an \texttt{Intel Core i7-1360P} processor (13\textsuperscript{th} generation) with {12} cores and {16} threads, and {16} GB of RAM. They were executed under \texttt{WSL 2} (version 2.6.3.0) using the \texttt{Ubuntu 22.04.5 LTS} distribution. The WSL environment had approximately \texttt{7.6} GB of memory available. For all experiments, no time limit was imposed. Each instance was therefore either solved to optimality or its resolution was aborted due to memory limitations.

All source code can be found in \citep{nicectp_graph_burning}. Due to the page limit the detailed results of our experiments have been omitted. They are nevertheless available in the same GitHub repository.

\subsection{Main Findings}

The correctness of the proposed ILP formulations and their implementations has also been verified empirically by comparing the optimal solutions obtained: for all solved instances, the optimal value of the maximum radius $\tau$ consistently matched $b(G)-1$.
We then evaluated the impact of the proposed techniques by comparing the performances obtained with the original model \eqref{fo_1}-\eqref{c_og_6} and the perturbed one defined by \eqref{eq:pert_obj}, \eqref{c_og_3}-\eqref{c_og_5}. For each, we tested 
\begin{enumerate}[i.]
    \item variable elimination by applying a preprocessing step based on \texttt{1DR} (see Theorem \ref{th:1dr});
    \item formulation reinforcement by applying \texttt{2DR} inequalities via user cuts (see Theorem \ref{th:2dr}); 
    \item pruning rule defined in Theorem \ref{th:prune} (only on perturbed model).
\end{enumerate} 

The preprocessing version of \texttt{1DR} eliminated, on average, approximately 82\% of all candidate variables, substantially reducing the model's size before optimization. These results correspond to approximately 83\% of the tested instances, since only 40 out of the 48 instances completed execution successfully, while the remaining instances were interrupted due to out of memory.

For the original model, approximately 23.0\% of the instances had the optimization aborted due to out of memory. Among the two tested techniques, the preprocessing version of \texttt{1DR} stands out as the most effective, solving approximately 81\% of the instances to optimality, improving performance on about 97\% of comparable instances, and reducing the number of CPLEX ticks by an average of 74.36\%. Only one instance showed degradation. The explicit addition of \texttt{2DR} as cuts had limited impact. While only approximately 65\% of the instances were solved to optimality, compared to the original model, approximately $90.3\%$ of the instances showed no measurable change, about $6.5\%$ exhibited marginal gains, and one instance showed a negligible degradation.


As for the perturbed model, the perturbation proved to be beneficial. Approximately 73\% of the instances were solved to optimality. However, two instances that were previously solved by the original model were terminated (\texttt{Killed}) in this configuration. Compared with the original model, the perturbed formulation improved performance on about $85.7\%$ of comparable instances, although a few became harder to solve due to an increased number of explored nodes. The pruning rule introduced for the perturbed model was also highly effective: compared to the perturbed model without any improvement technique, approximately $48.6\%$ of the instances showed improved performance, while $51.4\%$ remained unchanged, with no degradation observed. In several difficult cases, pruning drastically reduced the B\&B tree and even enabled the solution of the two instances that had previously been terminated.

The best overall configuration was obtained by combining the perturbed model, the pruning rule, and the preprocessing version of \texttt{1DR}, solving approximately 81\% of instances to optimality. This approach consistently outperformed the original model without any improvement technique across all comparable instances. In particular, it achieved a reduction in the number of CPLEX ticks in 100\% of the instances, demonstrating a uniform improvement with no performance degradation. On average, this configuration achieved a reduction of approximately 90.29\% in the number of CPLEX ticks.

Finally, we also tested a tighter upper bound for the admissible radii in the model. More specifically, we compared the model resolution using the upper bound for $b(G)$ from the heuristic of \citep{pereira2024graphburning}, instead of the bound derived from the Burning Number Conjecture. As expected, this tighter bound significantly improved performance on all instances, especially on larger ones. Indeed, a tigher bound directly implies less decision variables in the tested formulations. In several cases where the conjecture-based bound led to memory failure, the heuristic bound allowed the model to solve the instance successfully. However, the largest graphs with over 10,000 vertices still remained out of reach, which may be due to memory limitations in our computational setting.

\section{Conclusion}

This work introduced a new ILP formulation for the \textit{Graph Burning Problem}, based on covering the graph with burning balls of decreasing radii. Dominance rules and symmetry-breaking techniques were then proposed and implemented in order to reduce the search space of valid solutions. From a computational perspective, our results highlight two main conclusions. First, \texttt{1DR} is highly effective as a preprocessing technique: fixing redundant variables before optimization substantially reduces the search space and improves solver performance. Second, perturbing the objective function is particularly beneficial in this setting, especially when combined with the pruning rule proposed.
Our experiments also show that using a tighter heuristic upper bound for the admissible radii can be crucial on larger instances. In several cases, it turned otherwise unsolved instances into solvable ones. The overall performance was unequivocally enhanced, thereby demonstrating the robustness and relevance of the proposed methodologies.

Next steps include extending and/or lifting the family of \texttt{2DR} inequalities so that they may have a stronger impact on resolution performance. In order to deal with the lack of memory issues faced against very large instances, one could explore the idea of row generation, which has already proven its efficiency in the work of \citep{pereira2024graphburning}. Finally, the potential of \texttt{1DR} has only been partially explored for the moment. Indeed our work searches for dominated balls only as a preprocessing step. An interesting next step can be to consider this search within nodes of the Branch-and-Bound enumeration tree. This would allow to take advantage of the information issued from partial integer solutions and hence imply further variable fixings for children nodes.

~\\
\bibliographystyle{plainnat}
\bibliography{references}

\end{document}